\documentclass[12pt]{amsart}
\usepackage{fullpage, amsmath}
\usepackage{url}

\newtheorem{theorem}{Theorem}
\newtheorem{corollary}[theorem]{Corollary}
\newtheorem{lemma}[theorem]{Lemma}
\newtheorem{proposition}[theorem]{Proposition}
\newtheorem{conjecture}[theorem]{Conjecture}
\newtheorem*{conjecture-nonum}{Conjecture}

\theoremstyle{definition}
\newtheorem*{question}{Question}
\newtheorem*{problem}{Problem}

\begin{document}
\title{Words with many palindrome pair factors}

\author{Adam Borchert and Narad Rampersad}

\address{Department of Mathematics and Statistics \\
University of Winnipeg \\
515 Portage Avenue \\
Winnipeg, Manitoba R3B 2E9 (Canada)}

\email{adamdborchert@gmail.com, n.rampersad@uwinnipeg.ca}

\thanks{The first author is supported by an NSERC USRA, the second by an NSERC Discovery Grant.}


\date{\today}

\begin{abstract}
Motivated by a conjecture of Frid, Puzynina, and Zamboni,
we investigate infinite words with the property that for infinitely many $n$,
every length-$n$ factor is a product of two palindromes.  We show that
every Sturmian word has this property, but this does not characterize
the class of Sturmian words.  We also show that the Thue--Morse word
does not have this property.  We investigate finite words with the
maximal number of distinct palindrome pair factors and characterize
the binary words that are not palindrome pairs but have the property
that every proper factor is a palindrome pair.
\end{abstract}

\maketitle

\section{Introduction}
The \emph{palindromic length} of a word $x$ is the least $\ell$ such
that $x = p_1p_2\cdots p_\ell$, where each $p_i$ is a palindrome (we
consider the empty word a palindrome).  Frid, Puzynina, and Zamboni
\cite{FPZ13} made the following, remarkable conjecture:

\begin{conjecture-nonum}[Frid--Puzynina--Zamboni]
Let ${\bf w}$ be an infinite word.  If there exists $k$ such that
every factor of $w$ has palindromic length at most $k$, then ${\bf w}$ is
ultimately periodic.
\end{conjecture-nonum}

In this paper we focus on words with palindromic length at most
$2$:  these are called \emph{palindrome pairs} and have been studied
before \cite{GSS15,HM15,Kem82}.  Assuming that the
Frid--Puzynina--Zamboni Conjecture is true, there is no aperiodic word
with the property that all of its factors are palindrome pairs.  We
therefore investigate the following question, with the hope that it
will give some insight into the conjecture.

\begin{quotation}
Are there aperiodic words {\bf w} with the property ({\bf PP}) that for
\emph{infinitely many $n$}, every length-$n$ factor of ${\bf w}$ is a
palindrome pair?
\end{quotation}

As we will see, all Sturmian words have property {\bf PP}, but this does
not characterize the Sturmian words: there are other words with this
property as well.  We also show that the Thue--Morse word does not have
property {\bf PP}.  We then look at the analogue of (finite) \emph{rich words}
for palindrome pairs, rather than palindromes.  Finally, we
characterize the finite binary words that are not palindrome pairs, but have
the property that each of their proper factors is a palindrome pair.

\section{Palindrome pairs in Sturmian words}\label{sturmian}
In this section we show that every Sturmian word ${\bf s}$ has
property {\bf PP} and we characterize the lengths $n$ for which every
length-$n$ factor of ${\bf s}$ is a palindrome pair.  We begin with
some preliminaries concerning Sturmian words.

Let $\alpha$ be an irrational real number (fixed for the remainder of
this section) with $0<\alpha<1$ and let $\alpha$ have the continued
fraction expansion $\alpha = [0; a_1 + 1, a_2, a_3, \ldots]$.  We
define the following sequence of words:
\[
  s_0 = 1, \quad s_1 = 0, \quad s_n = s_{n-1}^{a_{n-1}} s_{n-2} \text{
  for } n \geq 2.
\]
	
Since $s_n$ is a prefix of $s_{n+1}$ for all $n\geq 1$, this sequence
converges to an infinite word ${\bf c}_\alpha$ called the
\emph{characteristic (or standard) Sturmian word} with slope $\alpha$.
An infinite word ${\bf s}$ is a \emph{Sturmian word} with slope
$\alpha$ if it has the same set of factors as ${\bf c}_\alpha$.  Since
we are only concerned with the language of factors of ${\bf s}$, we
will assume without loss of generality that ${\bf s} = {\bf c}_\alpha$.

Recall that words $x$ and $y$ are \emph{conjugates} if $x=uv$ and
$y=vu$ for some words $u$ and $v$.

	\begin{lemma}
		\label{pp_conj}
		Any conjugate of a palindrome pair is a palindrome pair.
	\end{lemma}	
	
	\begin{proof}
		Let $A_0 B_0$ be a palindrome pair where $A_0 = a_1 a_2 \cdots a_r$ and $B_0 = b_1 b_2 \cdots b_s$ are palindromes. As the result is trivial for the empty word, we may assume without loss of generality that $r \geq 1$ and $s \geq 0$. If $r = 1$, note that $B_0 a_1$ is a palindrome pair. Assume that $r\geq2$. Then, since $A_1 = a_2 a_3\cdots a_{r-1}$ and $B_1 = a_r b_1 b_2 \cdots b_s a_1$ are palindromes it follows that $A_1 B_1$ is a palindrome pair. The result follows on repeated application of this argument.
	\end{proof}

	\begin{lemma} \cite{F06}
		\label{all_pal_pref}
		A word $p$ is a palindromic prefix of ${\bf s}$ if and
                only if one of the following hold:
		\begin{itemize}
			\item $p$ is a factor of $s_1^{a_1}$.
			\item For either $u = p01$ or $u = p10$ we have $u = s_{n-1}^ks_{n-2}$ for some $n \geq 2$ and $1 \leq k \leq a_{n-1}$.
		\end{itemize}
	\end{lemma}
	
	\begin{lemma} \cite{dM94}
		\label{pal_pairs}
		For $n \geq 2$ and $1 \leq k \leq a_{n-1}$, $s_{n-1}^ks_{n-2}$ is a palindrome pair.
	\end{lemma}

If $w$ is a word (finite or infinite), the notation $w[i:j]$ indicates
the factor of $w$ beginning at position $i$ and ending at position $j-1$
in $w$.  Unless otherwise stated, we assume that positions are indexed
starting with $0$.

	\begin{lemma}
		\label{left_special}
		Let $w$ be a left special factor of ${\bf s}$ such that neither $w[0:|w|-1]$ nor $w$ is a palindrome. The words $1w$ and $0w$ are not both palindrome pairs.
	\end{lemma}
	
	\begin{proof}
		Let $w$ be such a factor. Since $w$ is left special,
                it is well known that $w$ occurs as a prefix of ${\bf
                  s}$. Assume without loss of generality that $1w$ is
                a palindrome pair. Write $1w = AB$ for palindromes $A$
                and $B$. Based on our hypotheses on $w$, it follows
                from Lemma \ref{all_pal_pref} that $|A| \ne 0,
                1$. Hence, we may write $A = 1P_11$ for a palindrome
                $P_1$. Since $P_1$ is a palindromic prefix of ${\bf
                  s}$ it then follows from Lemma \ref{all_pal_pref}
                that $B$ has prefix $0$. Then, since $B$ is a
                palindrome it follows that $B$ (and hence, $w$) has
                suffix $0$. If $0w$ were a palindrome pair it would
                follow by a similar argument that $w = 0p_201\ldots 1$
                and thus $w$ has suffix $1$, which is impossible.
	\end{proof}
	
	\begin{lemma}
		\label{exc_non_pp} 
		Let $u = s_{n-1}^ks_{n-2}$ for some $n \geq 4$ and $1
                \leq k \leq a_{n-1}$ where $|u| \geq 3 a_1 + 6$. There
                is a factor of ${\bf s}$ of length $|u| - 1$ which is not a palindrome pair.
	\end{lemma}
	
	\begin{proof}
		Let $a = 0^{a_1}$ and $b = 0^{a_1 + 1}$. Note that
                these are the only two maximal blocks of zeros
                occurring in ${\bf s}$. Let $w$ be the right special
                factor of ${\bf s}$ of length $|u| - 2$. Note that
                since $w0$ is a factor of ${\bf s}$, and since $b$ is
                maximal, that $w$ cannot have suffix $b$. Thus $w$ is
                a palindromic prefix of ${\bf s}$ with suffix (and
                hence, prefix) $a$. We consider the word $v =
                0^{-1}w01$. Note that $|v| = |u| - 1$. It follows from
                the facts that $w$ is right special and that $b$ is
                the maximal block of zeros in ${\bf s}$ that $v$
                occurs in ${\bf s}$. If $v$ is not a palindrome pair we are done. Assume otherwise, and write $v = P_1 P_2$ for palindromes $P_1$ and $P_2$.
		Suppose that $P_1 \ne 0^{-1} a$. It then follows that
		
		\[u = \overbrace{\ldots 10^{-1}a}^{P_1} \overbrace{1 \ldots}^{P_2}\]
		which contradicts that $a$ is the shortest maximal block of zeros in ${\bf s}$. Hence, $P_1 = 0^{-1}a$.
		
		Thus we have $w01 = aP_2 = a1 \ldots 1b1$. Since $P_2$
                is a palindrome it follows that $P_2 = 1b1\ldots
                1b1$. Then, since $w$ is a palindrome it follows that
                $w01 = aP_2 = a1b1\ldots 1b1b1$. Continuing in this
                manner we obtain that $P_2 = 1(b1)^k$ for some integer
                $k\geq0$. We now consider two cases depending on which
                of $w0$ or $w1$ occurs as a prefix of ${\bf s}$.
		
		\textbf{Case 1.} Assume that $w0$ is a prefix of ${\bf
                  s}$. Observe that $v$ occurs as a suffix of $s_n$,
                which itself is a prefix of ${\bf s}$. Let ${\bf
                  s}[i:j]$ be this occurrence of $v$. Since the prefix
                $s_n$ is either followed by $s_{n-1}$, or $s_n$ (which
                has prefix $s_{n-1}$) and since $n\geq 4$, we then have either
		
		\[{\bf s} = \overbrace{a1b1b\ldots\underbrace{\ldots
                    b1}_{{\bf s}[i:j]} }^{s_n} \overbrace{a1b \ldots}^ {s_{n-1}} \ldots\]
		
		or
		
		\[{\bf s} = \overbrace{a1b1b\ldots\underbrace{\ldots
                    b1}_{{\bf s}[i:j]} }^{s_4} \overbrace{a10}^ {s_{3}} \overbrace{a1} ^{s_{2}} \ldots\]
		
		depending on the length of $s_{n-1}$.
		
		Now, let $x = {\bf s}[i+2|b|:j+2|b|]$. Note that $|x|
                = |v| = |u| - 1$. Suppose for contradiction that $x$
                is a palindrome pair. Note that $x$ has prefix $b$ and
                suffix $1a1b$ where $1a1$ is unioccurrent in
                $x$. Since $x$ is a palindrome pair it then follows
                that $x$ is a palindrome, and hence that $x =
                b1a1b$. Thus, $|u| = 3|b| + 2 < 3 a_1 + 6$, contrary
                to our hypothesis.
		
		\textbf{Case 2.} We now assume that $w1$ is a prefix
                of ${\bf s}$. Note that $w1 = a1(b1)^k a1$ for some $k
                > 0$, and that $w10 = u$. As before, $u$ is either
                followed by $s_{n-1}$, or $s_n$ (which has prefix
                $s_{n-1}$). Thus, \[{\bf s} = \overbrace{ \underbrace {a1b1b\ldots b1a} _ {w}10}^{u} \overbrace{a1\ldots}^{s_{n-1}}\ldots\]
		
		Set $x = {\bf s}[|b|:|u| + |a|]$. Again, note that $|x| = |u| - 1$. We suppose that $x$ is a palindrome pair and apply the same argument to $x$ as was done in the previous case to obtain $x = b1a1b$ and thus that $|u| < 3 a_1 + 6$, a contradiction. This completes the proof.
		
	\end{proof}
	
	\begin{lemma}
		\label{stur_non_pps}
		 If $n \ne |s_{m-1}^ks_{m-2}|$ for any $m \geq 2$ and
                 $1 \leq k \leq a_{m-1}$ where $n \geq 3 a_1 + 6$,
                 then ${\bf s}$ has at most $n-1$ palindrome pair factors of length $n$.
	\end{lemma}
	
	\begin{proof}
		This follows from Lemmas \ref{left_special} and
                \ref{exc_non_pp}, the fact that ${\bf s}$ has factor
                complexity $n+1$ and the fact that the factor set of
                ${\bf s}$ is closed under reversal.
	\end{proof}
	
	\begin{theorem}
		\label{sturm_pals}
		Let ${\bf s}$ be a Sturmian word and let $n
                \geq 3 a_1 + 6$ be a positive integer. Every factor of
                ${\bf s}$ of length $n$ is a palindrome pair if and only if $n = |s_{m-1} ^ k s_{m-2}|$ for some $m \geq 4$ and $1 \leq k \leq a_{m - 1}$.
	\end{theorem}
	
	\begin{proof}
		Let $n$ be such an integer. It is known that the $n+1$
                factors of length $n$ in ${\bf s}$ are the conjugates
                of $s_{m-1} ^ k s_{m-2}$, as well as a palindromic
                singular factor \cite{Mel99}. It follows from Lemma
                \ref{pal_pairs} that $s_{m-1} ^ k s_{m-2}$ is a
                palindrome pair. It then follows from Lemma
                \ref{pp_conj} that all factors of length $n$ are
                palindrome pairs. The converse follows immediately
                from Lemma \ref{stur_non_pps}.
	\end{proof}

\begin{corollary}
Every Sturmian word has property {\bf PP}.
\end{corollary}

Property {\bf PP} does not characterize Sturmian words.  Consider the
following construction.  For $i = 1,2,\ldots$ let $w_i$ be an
arbitrary palindrome of some length $t_i$ over $\{0,1\}$ and write
$w_i = w_{i,1}w_{i,2}\cdots w_{i,t_i}$.  We define a sequence of palindromes
$p_1,p_2,\ldots$ as follows:
\begin{align*}
p_1 &= 0 \\
p_2 &= p_1\, w_{1,1}\, p_1\, w_{1,2} \cdots p_1\, w_{1,t_1}\, p_1 \\
p_3 &= p_2\, w_{2,1}\, p_2\, w_{2,2} \cdots p_2\, w_{2,t_2}\, p_2 \\
&\vdots
\end{align*}
Note that $p_i$ is a prefix of $p_{i+1}$ for all $i$, so we have a
limiting infinite word ${\bf p}$.

\begin{proposition}\label{shuffling_pals}
For each $i \geq 1$, every factor of ${\bf p}$ of length $|p_i|+1$ is
a palindrome pair.
\end{proposition}

\begin{proof}
Any factor of ${\bf p}$ of length $|p_i|+1$ is a conjugate of $p_ia$
for some $a \in \{0,1\}$.  The result follows from
Lemma~\ref{pp_conj}.
\end{proof}

Now the word ${\bf p}$ is not necessarily Sturmian.  For a suitable
choice of $w_i$, we have that $0p_i0$ and $1p_i1$ are both factors of
${\bf p}$, which means that ${\bf p}$ is not \emph{balanced}, and hence not
Sturmian.

\section{Palindrome pairs in the Thue--Morse word}\label{tm}
Let $\mu$ be the morphism that maps $0\to 01$ and $1\to 10$.
Let ${\bf t} = \mu^\omega(0)$ be the Thue--Morse word.

\begin{theorem}\label{tm_pp}
The word ${\bf t}$ does not have property {\bf PP}.
\end{theorem}

\begin{proof}
This can be verified ``automatically'' using the Walnut Prover
software created by Hamoon Mousavi and available at
\begin{center}
\url{https://cs.uwaterloo.ca/~shallit/papers.html}~.
\end{center}
We will not explain in detail the methodology implemented by the
Walnut Prover (the reader may consult \cite{GHS13} or \cite{SS15}, for example).  The
important point is that to be able to apply this method, the property
{\bf PP} needs to be expressible in a certain extension of first order
logic.  The relevant formulae are given below.  The first defines all
pairs $(i,j)$ such that ${\bf t}[i : j]$ is a palindrome:
\[
(i,j) : (i=j) \vee ((i<j) \wedge \forall k\; (i+k < j) \Rightarrow
{\bf t}[i+k] = {\bf t}[j-1-k]).
\]
Let us refer to this formula as $\text{palindrome}(i,j)$.
The clause $(i=j)$ corresponds to the empty palindrome when used in
the next formula, which defines all $n$ such that every factor of
${\bf t}$ is a palindrome pair:
\[
(n) : \forall i \; \exists j \; \text{palindrome}(i,j) \wedge
\text{palindrome}(j,i+n).
\]
The output of the Walnut Prover is a description (by a finite
automaton) of the binary representations of all $n$ defined by the
formula above.  After running the prover we find that for $0 \leq n
\leq 5$, every length-$n$ factor of ${\bf t}$ is a palindrome pair,
but for $n \geq 6$ there is at least one length-$n$ factor of ${\bf
  t}$ that is not a palindrome pair.
\end{proof}

\begin{proposition}
For each $k \geq 1$ there are at least $3\cdot 2^{2k+1}$ factors of ${\bf
  t}$ of length $3\cdot 2^{2k}$ that are palindrome pairs.
\end{proposition}

\begin{proof}
It is well known that for each $k$ the squares $\mu^k(010)\mu^k(010)$
and $\mu^k(101)\mu^k(101)$ occur in the Thue--Morse word.
Furthermore, since $\mu^2(0)=0110$
and $\mu^2(1) = 1001$, we see that $\mu^{2k}(010)$ and $\mu^{2k}(101)$
are palindromes.  Thus every conjugate of $\mu^{2k}(010)$ and
$\mu^{2k}(101)$ is a product of two palindromes and occurs in ${\bf t}$.
\end{proof}

\begin{proposition}
For odd $n$ there are at most 32 factors of ${\bf t}$ of
length $n$ that are palindrome pairs.
\end{proposition}

\begin{proof}
Let $w$ be an factor of ${\bf t}$ of odd length.  If $|w| \in \{1,3\}$
the result is clear, so suppose $|w| \geq 5$.  Blondin-Mass\'e et
al.~\cite{BBGL08} gave the following formula for the palindromic
complexity $P_{\bf t}(n)$ of ${\bf t}$:
\begin{equation}\label{tm_pal_cmplxty}
P_{\bf t}(n) =
\begin{cases}
1 & n = 0,\\
2 & 1 \leq n \leq 4,\\
0 & n \text{ odd and } n \geq 5,\\
4 & n \text{ even and } 4^k+2 \leq n \leq 3\cdot 4^k, \text{ for } k
\geq 1,\\
2 & n \text{ even and } 3\cdot 4^k+2 \leq n \leq 3\cdot 4^{k+1}, \text{ for } k
\geq 1.
\end{cases}
\end{equation}
Consequently $w$ is not a palindrome.  Suppose $w$ is a product of two
palindromes, $u$ and $v$, where $|u|$ is odd.
We have two choices for $|u|$, namely $|u| \in
\{1,3\}$, and for each choice of length there are two possible values
for $u$.  Thus there are four choices for $u$.  By
\eqref{tm_pal_cmplxty} there are at most four choices for $v$.  Since
we could have $w=uv$ or $w=vu$, this gives at most $32$ possibilities
for $w$.
\end{proof}

\section{Words with the maximal number of palindrome pair factors}\label{rich_pp}
A word $w$ of length $n$ is called \emph{rich} (i.e., ``rich in
palindromes'') if it contains the maximum number (viz.\ $n+1$) of
distinct palindromic factors \cite{GJWZ09}.  Rich words have been
extensively studied.  Here we investigate the analogue of richness for
palindrome pairs; i.e., we study words of length $n$ that have the
maximum possible number of distinct factors that are palindrome pairs.

As it turns out, we are unable to prove anything concerning this
problem; we therefore will just present some conjectures based on
computer calculations.  Computer experiments suggest that for $n\geq
1$ the maximum possible number of distinct palindrome pair factors in
a binary word of length $n$ is
\[
\left\lceil \frac{n^2+2n+3}{3} \right\rceil.
\]
Furthermore, computer experiments also suggest that if $x$ is a binary
word of length $n$ beginning with $0$ that contains the maximum number
of distinct palindrome pair factors, then
\begin{itemize}
\item if $n \equiv 0 \pmod 3$, then $x$ is one of
\[
\begin{array}{lll} 0^{n/3+1}1^{n/3}1^{n/3-1},& 0^{n/3}1^{n/3}1^{n/3},&
0^{n/3}1^{n/3+1}1^{n/3-1}, \\ 0^{n/3-1}1^{n/3}1^{n/3+1},&
0^{n/3-1}1^{n/3+1}1^{n/3};&
\end{array}
\]
\item if $n \equiv 1 \pmod 3$, then $x$ is one of
\[
\begin{array}{lll} 0^{\lceil n/3 \rceil}1^{\lfloor n/3 \rfloor}1^{\lfloor n/3
  \rfloor},& 0^{\lceil n/3 \rceil}1^{\lceil n/3 \rceil}1^{\lfloor n/3
  \rfloor-1},& 0^{\lfloor n/3 \rfloor}1^{\lfloor n/3 \rfloor}1^{\lceil
  n/3 \rceil},\\ 0^{\lfloor n/3 \rfloor}1^{\lceil n/3 \rceil}0^{\lfloor
  n/3 \rfloor},& 0^{\lfloor n/3 \rfloor-1}1^{\lceil n/3 \rceil}1^{\lfloor n/3
  \rfloor};&
\end{array}
\]
\item if $n \equiv 2 \pmod 3$, then $x$ is one of
\[
0^{\lceil n/3 \rceil}1^{\lceil n/3 \rceil}1^{\lfloor n/3
  \rfloor}, \quad 0^{\lfloor n/3 \rfloor}1^{\lceil n/3 \rceil}0^{\lceil n/3
  \rceil}.
\]
\end{itemize}
Given how simple these words are, it would be nice to prove that they
do indeed contain the maximum number of distinct palindrome pair
factors. 

\section{Minimal non palindrome pairs}
In this section we look at words somewhat related to those described
in the previous section.  A \emph{minimal non palindrome pair} is a
word that is not a palindrome pair but has the property that each of
its proper factors is a palindrome pair.  It turns out that we can
characterize the minimal non palindrome pairs over the alphabet
$\{0,1\}$.

Given a word $w$, a \emph{block} of $w$ is an occurrence of a factor
of $w$ that consists of one or more repetitions of a single letter and
that cannot be extended to either side to create a longer such occurrence.
The following lemma is easily verified and may be used without reference in the following results.

\begin{lemma}
	\label{3_block_all_pp}
	Any word with at most three blocks is a palindrome pair.
\end{lemma}

A block that is neither the first nor the last block in a word $w$ is
called an \emph{internal block}.  A block $b$ of $w$ is a
\emph{maximum block} if all other blocks have length at most the
length of $b$.  Among all maximum blocks, the internal ones are called
\emph{internal maximum blocks}.
We now state and prove a critical lemma.

\begin{lemma}
	\label{max_5_blocks}
	A minimal non palindrome pair with an internal maximum block has at most five blocks.
\end{lemma}

\begin{proof}
	Let $w$ be a minimal non palindrome pair with an internal maximum block $x$ (of ones, say). Note that $|x| \geq 2$. Suppose that $w$ has at least six blocks. Since $x$ is internal and $w$ has at least six blocks we may without loss of generality assume that
	
	\[z_0 = 0 \ldots 0 \underbrace{1 \ldots 1} _ {x} 0 \ldots 0 1 \ldots 1\] 
	is a four block factor of $w$ containing $x$ internally. Since 
	
	\[y = 0\ldots 0 \underbrace{1\ldots 1} _ {x} 0\ldots 0 1\]
	is a proper factor of $w$, and hence a palindrome pair, and since $|x|\geq 2$, it follows that the blocks of $0$'s in $z_0$ have equal length. Then, since $0^{-1} y$ is also a palindrome pair it follows that these blocks of $0$'s have length $1$.
	
	We now consider two cases.
	
	\textbf{Case 1.} Assume that $z_0$ is not a suffix of $w$. Then, the factor 
	
	\[z_1 = 0 \underbrace{1\ldots 1} _ {x} 0 1 \ldots 1 0\]	
	occurs in $w$. Since $w$ has at least six blocks it follows that $z_1$ is a proper factor of $w$, and hence a palindrome pair. It then follows that $z_1$ is a palindrome. Before we proceed, we show that the other case results in the same occurrence.
	
	\textbf{Case 2.} Assume now that $z_0$ is a suffix of $w$. Since $w$ has at least six blocks it then follows that 
	
	\[z_2 =0 1\ldots 1 0 \underbrace{1\ldots 1} _ {x} 0\]	
	is a proper factor of $w$, and thus a palindrome pair, and
        hence a palindrome.
	
In either case, we find that $w$ contains a palindromic factor of the
form $z_1$.

	We now claim that all blocks of zeros in $w$ have length
        $1$. Suppose to the contrary that $00$ occurs in $w$. Without
        loss of generality we may assume that there is an occurrence
        of $00$ that begins prior to the start of $z_1$.  Let $b_0$ be
        the rightmost such occurrence. Let $u_0$ be the proper factor
        of $w$ that begins with this occurrences of $b_0$ and ends
        with the last block of ones in $z_1$. Since $u_0$ is a proper factor, and hence a palindrome pair, it follows that ${b_0}^{-1} u_0$ is a palindrome. However, we then have that the proper factor $u_0 1^{-1}$ of $w$ is not a palindrome pair, a contradiction. Hence $00$ does not occur in $w$.
	
        Next, we show that all internal blocks of ones in $w$ are
        maximum blocks. Suppose that there is a block of ones in $w$
        which is not a maximum block. Let $b_1$ be an occurrence of
        this with minimum possible distance to $z_1$ in $w$. Without
        loss of generality we may assume that $b_1$ precedes $z_1$ in
        $w$. Then, since $b_1$ is internal and has minimum possible
        distance to $z_1$ in $w$, it follows that
        \[0 b_1 0 \underbrace{1\ldots 1} _ {|x|} 0\]
        is a proper factor of $w$. This factor is not a palindrome
        pair, a contradiction. Hence, all internal blocks of ones in
        $w$ are maximum blocks.
	
	It follows from these two properties (all blocks of zeros in
        $w$ have length 1 and all internal blocks of ones in $w$ have
        the same length) that $w$ is a palindrome pair. This is easily
        verified by considering the initial and final blocks of
        $w$. Hence, we have a contradiction. This completes the proof.	
\end{proof}

We now make use of this lemma to prove two results which will allow us
to construct the class of minimal non palindrome pairs.  By
``extending a block'' we mean increasing its length by 1 or more.

\begin{lemma}
	\label{mnpp_add}
	The word resulting from extending an internal maximum block in a minimal non palindrome pair is itself a minimal non palindrome pair.
\end{lemma}

\begin{proof}
	Clearly, the result will follow on repeated application if we
        show it holds when extending an internal maximum block by
        one. Let $w$ be a minimal non palindrome pair with an internal
        maximum block $x$. Note $|x| \geq 2$. Let $w'$ be the word
        resulting from extending $x$ by one and let $x'$ be the extension of $x$ in $w'$.
	
	We first show that $w'$ is not a palindrome pair. Suppose to
        the contrary and write $w' = A'B$ for palindromes $A'$ and
        $B$. Assume without loss of generality that $A'$ contains at
        least half of $x'$. We may then write $w = AB$ where $A$ is
        $A'$ with one symbol of $x'$ removed. Suppose for
        contradiction that $x'$ is contained in $A'$. Since $x$ was a
        maximum block in $w$ it follows that $x'$ is unioccurrent in
        $w'$. Hence $x'$ is centred in $A'$. Clearly in this case $w$
        is also a palindrome pair, a contradiction. We conclude that
        $x'$ is split between $A'$ and $B$. It then follows from a
        similar argument that neither $A'$ nor $B$ can be contained in
        $x'$. It now follows from Lemma \ref{max_5_blocks} applied to
        $w$ and the fact that $A'$ and $B$ are palindromes that
	
	\[
	w' = \lefteqn{
	\underbrace{\phantom{11\ldots 10\ldots 011\ldots 1}}_{A'}}11\ldots 10\ldots 0 
	\lefteqn{\overbrace{\phantom{11\ldots1 1\ldots 1}}^{x'}} 11\ldots1 
	\underbrace{1\ldots 10\ldots 0 1\ldots 1}_B
	\] and hence
	
	\[
	w = \lefteqn{
		\underbrace{\phantom{11\ldots 10\ldots 01\ldots 1}}_A}11\ldots 10\ldots 0 
	\lefteqn{\overbrace{\phantom{1\ldots1 1\ldots 1}}^{x}} 1\ldots1 
	\underbrace{1\ldots 10\ldots 0 1\ldots 1}_B
	\]	
	where $11$ is a prefix of $w'$, and hence also of $w$. Note that since $w$ is not a palindrome pair, the blocks of zeros in $A$ and $B$ cannot be equal. Let $u$ be the prefix of $w$ ending with the final zero of $B$. Since $w$ has prefix $11$, it follows that $u$ and $1^{-1} u$ are not both palindrome pairs. Since these are proper factors of $w$, we have a contradiction. This completes the proof that $w'$ is not a palindrome pair.
	
	We must now prove the minimality of $w'$. Suppose for contradiction that $w'$ has a proper factor $u'$ which is not a palindrome pair. Let $u$ be the corresponding factor in $w$ (note that $u' = u$ unless $u'$ contains $x'$). Since $u$ is a proper factor of $w$, it is a palindrome pair. Write $u = ab$ for palindromes $a$ and $b$. Since $u'$ is not a palindrome pair it follows that $x$ must be contained in $u$, but cannot be centred in either $a$ or $b$. Without loss of generality assume that at least half of $x$ is in $a$. We now consider two cases.
	
	\textbf{Case 1.} Suppose that $x$ is contained in $a$. Without
        loss of generality we then have by Lemmas~\ref{3_block_all_pp}
        and \ref{max_5_blocks} that either
	\[u = \underbrace{1\ldots1 0\ldots0 \overbrace{1\ldots1}^x } _ a \underbrace{0\ldots0} _ b\] 
	or
	\[u = \underbrace{0\ldots 0 1\ldots1 0\ldots0 \overbrace{1\ldots1}^x 0\ldots 0} _ a \underbrace{0\ldots0} _ b\]	
	First, assume the former. Since $1^{-1} u$ is a proper factor of $w$, and hence a palindrome pair, it follows that the blocks of zeros in $a$ and $b$ are equal. This contradicts the assumption that $u'$ is not a palindrome pair.
	
	In the latter case, Lemma \ref{max_5_blocks} implies that $w$
        consists of 5 blocks.  Therefore, $w$ can be obtained by
        adding some number of zeros to the beginning and end of $u$.
        However, adding such zeros to the ends of $u$ preserves the
        property of being a palindrome pair.  Consequently, $w$ is a
        palindrome pair, a contradiction. This completes the case.
	
	\textbf{Case 2.} Suppose that $x$ is split between $a$ and $b$. Since $u'$ is not a palindrome pair and by Lemma \ref{max_5_blocks} we may write
	\[
	u = \lefteqn{
		\underbrace{\phantom{11\ldots 10\ldots 01\ldots 1}}_a}11\ldots 10\ldots 0 
	\lefteqn{\overbrace{\phantom{1\ldots1 1\ldots 1}}^{x}} 1\ldots1 
	\underbrace{1\ldots 10\ldots 0 1\ldots 1}_b
	\]		
	where the blocks of zeros in $a$ and $b$ are not equal. Let $y$ be the prefix of $u$ ending with the last zero of $b$. Note that $y$ and $1^{-1}y$ cannot both be palindrome pairs, but are both proper factors of $w$. This is a contradiction, which completes the proof.
\end{proof}

\begin{lemma}
	\label{mnpp_delete}
	The word resulting from deleting one digit from an internal maximum block of a minimal non palindrome pair of length at least $7$ is itself either a minimal non palindrome pair or is one of $01^n01^n0$ or $10^n10^n1$ for some integer $n\geq1$.
\end{lemma}

\begin{proof}
	Let $w'$ be a minimal non palindrome pair of length at least
        $7$ with an internal maximum block $x'$ (of ones, say). Note
        $|x'| \geq 2$. Let $w$ and $x$ be obtained from $w'$ and $x'$
        respectively by deleting one digit from $x'$.  If $w$ is of
        the form $01^n01^n0$ we are done, so suppose this is not the case.
	
	We first prove that $w$ is not a palindrome pair. Suppose to
        the contrary and write $w = AB$ for palindromes $A$ and
        $B$. Without loss of generality assume that at least half of
        $x$ is in $A$. Since $w'$ is not a palindrome pair it follows
        that $x$ is not centred in either $A$ or $B$, nor is either $A$ or $B$ contained in $x$. There are two cases to now consider.
	
	\textbf{Case 1.} Assume that $x$ is contained in $A$. By
        Lemmas~\ref{3_block_all_pp} and \ref{max_5_blocks} there are without loss of generality two further possibilities. Either
	\[w = \underbrace{1\ldots1 0\ldots0 \overbrace{1\ldots1}^x } _ A \underbrace{0\ldots0} _ B\]
	or
	\[w = \underbrace{0\ldots 0 1\ldots1 0\ldots0 \overbrace{1\ldots1}^x 0\ldots 0} _ A \underbrace{0\ldots0} _ B\]
	
	First, assume the former. Since $w'$ is not a palindrome pair it follows that the blocks of zeros in $w'$ are not equal. Since $1^{-1}w'$ is a proper factor of $w'$, and hence a palindrome pair, it follow that $w' = 10\ldots 11 0\ldots0$. Since $x'$ is maximum in $w'$ it then follows that $|w'| = 6$, a contradiction.
	
	Next, assume the latter. Note that since $w'$ is not a
        palindrome pair, the blocks of ones in $w'$ have different
        lengths. Let $v$ be the suffix of $w'$ starting with the first
        one of $w'$. Since $v$ is a proper factor of $w'$, and hence a
        palindrome pair, it follows that the blocks of zeros in $v$
        are equal. Applying a similar argument to the first four
        blocks of $w'$ we get that the first two blocks of zeros in
        $w'$ are also equal. Since $A$ is a palindrome it then follows
        that $B = \varepsilon$. Since $0^{-1}w'$ is a palindrome pair
        it follows that the blocks of zeros in $w'$ have length
        $1$. Thus $w' = 0 1^{|x'|-1}0 1^ {|x'|}0$, and hence $w = 0
        1^{|x'|-1}0 1^ {|x'|-1}0$, which is a contradiction.  This completes the case.
	
	\textbf{Case 2.} Assume $x$ is split between $A$ and $B$. By Lemma \ref{max_5_blocks} it follows that 
	\[
	w = \lefteqn{ \underbrace{ \phantom{1 \ldots 10\ldots 01 \ldots 1 }}_A} 1\ldots 10\ldots 0 
	\lefteqn{\overbrace{\phantom{1\ldots1 1\ldots 1}}^{x}} 1\ldots1 
	\underbrace{1\ldots 10\ldots 0 1\ldots 1}_B
	\]
	
	Since $w'$ is not a palindrome pair it follows that the blocks of zeros in $w'$ are not equal. Thus, the proper factor of $w'$ obtained from $w'$ by deleting its final block of ones is not a palindrome pair. This contradiction completes the proof that $w$ is not a palindrome pair.
	
	It remains to prove that $w$ is minimal. Suppose that some proper factor $u$ of $w$ is not a palindrome pair. Let $u'$ be the corresponding proper factor of $w'$. Since $u'$ is a palindrome pair, while $u$ is not, $u$ must contain $x$. Write $u' = ab$ for palindromes $a$ and $b$. Without loss of generality we may assume that at least half of $x'$ is in $a$. Clearly, $x'$ is not centred in either $a$ or $b$. By Lemma \ref{max_5_blocks} without loss of generality we again consider two cases.
	
	\textbf{Case 1.} Assume that $x'$ is contained in $a$. Then either 
	\[u' = \underbrace{1\ldots1 0\ldots0 \overbrace{1\ldots1}^{x'} } _ a \underbrace{0\ldots0} _ b\]
	or
	\[u' = \underbrace{0\ldots 0 1\ldots1 0\ldots0 \overbrace{1\ldots1}^{x'} 0\ldots 0} _ a \underbrace{0\ldots0} _ b\]
	
	First, assume the former. Since $u$ is not a palindrome pair it follows that the blocks of zeros in $u'$ are not equal. Thus $1^{-1}u'$ is a proper factor of $w'$ which is not a palindrome pair, a contradiction.
	Next, assume the latter. Lemma \ref{max_5_blocks} implies that
        $w'$ is obtained from $u'$ by adding some number of zeros to
        the beginning and end of $u'$.  However, this would mean that
        $w'$ is a palindrome pair, which is a contradiction. This contradiction completes the proof.
\end{proof}

One final Lemma is needed for our main result.

\begin{lemma}
	\label{mnpp_without_max_int_blk}
	The only minimal non palindrome pairs which do not have internal maximum blocks are $11(01)^n00$ and $00(10)^n11$ for each $n\geq1$.
\end{lemma}

\begin{proof}
	Clearly, these words are minimal non palindrome pairs. Let $w$
        be a minimal non palindrome pair with no internal maximum
        blocks that is not one of the words given in the statement of the lemma. From left to right, let $b_i$ be the $i$th block of ones in $w$. Assume without loss of generality that $b_0$  is both a prefix and a maximum block of $w$. We prove that $w$ does not have an internal block of zeros with length at least $2$. Suppose to the contrary that such a block exists. Let $u$ be the shortest prefix of $w$ with suffix $00$. Since $u$ is a proper factor of $w$, and hence a palindrome pair, it then follows $u= b_0 00$. Since $w$ is not a palindrome pair it follows that $w$ has at least four blocks. Let $v$ be the smallest four block prefix of $w$. We have $v = b_0 0\ldots 0 b_1 0$, which is not a palindrome pair. Thus, $v = w$. Now, since $1^{1-|b_0|}v$ is a proper factor of $w$, and hence a palindrome pair, it follows that $|b_1| = 1$. Finally, since $1^{-1}v$ is a proper factor of $w$, and hence a palindrome pair, it follows that $b_0 = 11$, and hence since $b_0$ is maximum and $w$ is not a palindrome pair, that $w = 110100$. We may conclude that $w$ has no internal block of at least two zeros. We now consider two cases.
	
	\textbf{Case 1.} Assume $w$ has at least five blocks. Thus,
        $y_0 = b_0 0 b_1 0 b_2$ is a prefix of $w$. Since $w$ is not a
        palindrome pair it follows that $y_1 = y_0 0$ is also a prefix
        of $w$. Since ${b_0}^{-1} y_1$ is a proper factor of $w$, and
        hence a palindrome pair, it follows that $|b_1| = |b_2|$.
        Since $w$ is not a palindrome pair it then follows that
        $y_2 = y_1 b_3 0$ is a prefix of $w$. We then argue as before
        to show that $|b_3| = |b_2|$, and continue until we reach the
        final block of $w$. Since $w$ is not a palindrome pair, we
        have that $w = b_0 0 b_1 0 \ldots b_k 0\ldots 0$ for some
        integer $k\geq1$ (note that $00$ is a suffix of $w$ and
        $|b_i| = |b_j|$ for any $1\leq i < j \leq k$). If $|b_k| >1$,
        then the suffix of $w$ starting with the last $1$ of $b_{k-1}$
        is a proper factor of $w$, but not a palindrome pair, a
        contradiction. Thus all internal blocks of ones have length
        $1$. Thus $z = 11(01)^r 00$ is a factor of $w$ for some
        $k\geq1$, and hence, as $z$ is not a palindrome pair, $z =
        w$. This completes the case.
	
	\textbf{Case 2.} Assume now that $w$ has at most four blocks. Thus \[w = b_0 0 b_1 0\ldots 0\]
	
	Since $w$ is not a palindrome pair it follows that $00$ is a suffix of $w$. Since $10 b_1 00$ is a proper factor of $w$, and hence a palindrome pair, it follows that $|b_1| = 1$. Thus $w$ contains, and is hence equal to, $110100$. This completes the proof.
	
\end{proof}

We now construct the class of minimal non palindrome pairs inductively. It is easily verified that every word of length five or less is a palindrome pair. The minimal non palindrome pairs of length six are given below. For $i\geq 7$, to generate the minimal non palindrome pairs of length $i$:
\begin{itemize}
	\item Extend any internal maximum block in a minimal non palindrome pair of length $i-1$ by one.
	
	\item If $i$ is even, add the words
          $11(01)^{\frac{i-4}{2}}00$, $0 1^\frac{i-4}{2} 0
          1^{\frac{i-4}{2}+1} 0$, $1 0^\frac{i-4}{2} 1
          0^{\frac{i-4}{2}+1} 1$, $01^{\frac{i-2}{2}}0^{\frac{i-2}{2}}1$ and their reverses.
\end{itemize}

The minimal non palindrome pairs for the first few $i$ are given in the following table. For conciseness, reverses and words with prefix $1$ have been omitted.

\begin{table}[h!]
\begin{center}
\begin{tabular}{|l||c|c|c|c|c|}
\hline
	i & 6 & 7 & 8 & 9 & 10 \\ \hline
	  & 001011 & 0011101 & 00101011 & 001111101 & 0010101011\\
	  & 001101 & 0100011 & 00111101 & 010000011 & 0011111101\\
	  & 010011 & 0101110 & 01000011 & 010111110 & 0100000011\\
	  & 010110 & 0110001 & 01011110 & 011000001 & 0101111110\\
	  & 011001 & 0111001 & 01100001 & 011011110 & 0110000001\\
	  &		   &   		 & 01101110 & 011100001 & 0110111110\\
	  &		   &   		 & 01110001 & 011110001 & 0111000001\\
	  &		   &   		 & 01111001 & 011111001 & 0111011110\\
	  &		   &		 &		    &			& 0111100001\\
	  &		   &		 &		    &			& 0111110001\\
	  &		   &		 &		    &                  & 0111111001\\
\hline
\end{tabular}
\caption{List of short minimal non palindrome pairs}
\end{center}
\end{table}

This leads us to our main result.

\begin{theorem}
	The minimal non palindrome pairs are exactly those words described above.
\end{theorem}

\begin{proof}
  It is easily verified that those words of length six given in the
  table, as well as those described in part two of the construction
  above are minimal non palindrome pairs. It then follows from Lemma
  \ref{mnpp_add} that all the constructed words are also minimal non
  palindrome pairs.
	
  To show that these are the only minimal non palindrome pairs, let
  $w'$ be a minimal non palindrome pair of length $i \geq 7$.  If $w'$
  has no internal maximum blocks, then by
  Lemma~\ref{mnpp_without_max_int_blk}, $w'$ is either
  $11(01)^{\frac{i-4}{2}}00$ or its reverse, and is therefore
  accounted for by our construction.

  If $w'$ has exactly one internal maximum block, then $w'$ can be
  obtained by extending an internal maximum block in some word $w$ of
  length $i-1$.  By Lemma~\ref{mnpp_delete} $w$ is either a minimal
  non palindrome pair, or is one of
  $0 1^\frac{i-4}{2} 0 1^{\frac{i-4}{2}} 0$ or
  $1 0^\frac{i-4}{2} 1 0^{\frac{i-4}{2}} 1$.  Again $w'$ is accounted
  for by our construction.

  If $w'$ has more than one internal maximum block, then without loss
  of generality $w'$ contains a factor $01^k0^k1$ for some $k \geq 2$.
  However this factor is not a palindrome pair.  It follows that this
  factor is not a proper factor and hence that
  $w' = 01^{\frac{i-2}{2}}0^{\frac{i-2}{2}}1$.  Again $w'$ is accounted
  for by our construction.  This concludes the proof.
\end{proof}

\begin{corollary}
	Let $npp(i)$ be the number of minimal non palindrome pairs of length $i$. Then
	\[npp(i) = \begin{cases}
	0 & i<6\\
	4i - 12 & i\geq 6, ~i \text{ even}\\
	npp(i-1) & i\geq6, ~i \text{ odd}
	\end{cases}\]
\end{corollary}

\begin{proof}
One first checks that there are 12 minimal non palindrome pairs of
length 6.  Let $i > 6$ be odd.  Every minimal non palindrome pair
of length $i-1$ with exactly 1 internal maximum block produces one of
length $i$ by extending this internal maximum block.  The word
$11(01)^{\frac{i-5}{2}}00$ and its reversal produce no minimal non
palindrome pairs of length $i$, since they have no internal maximum
blocks to extend.  The word $01^{\frac{i-3}{2}}0^{\frac{i-3}{2}}1$ and
its reversal each produce two minimal non palindrome pairs, since they
each have two internal maximum blocks.  Thus there is no net increase
in the number of minimal non palindrome pairs when going from length
$i-1$ to $i$, and so we have $npp(i)=npp(i-1)$.

On the other hand, if $i$ is even, then every minimal non palindrome
pair of length $i-1$ produces one minimal non palindrome pair of length
$i$.  Additionally, our construction adds eight new words of length
$i$.  Thus there are
\begin{eqnarray*}
&&npp(i-1)+8 \\
&=& npp(i-2)+8 \quad\quad\text{(since $i-1$ is odd)}\\
&=& 4(i-2)-12+8 \quad\quad\text{(inductively)}\\
&=& 4i-12
\end{eqnarray*}
minimal non palindrome pairs of length $i$.
\end{proof}

\section{Conclusion}
Here we mention some interesting open questions raised by the previous
results.  The first one is an obvious one, although we shall
see shortly that the answer may, in fact, be ``no''.

\begin{question}
Does property {\bf PP} characterize some interesting class of words?
\end{question}

We can also define a complexity function based on palindrome pairs.
Recall that the \emph{factor complexity function} $C_{\bf w}(n)$
counts the number of factors of ${\bf w}$ of length $n$.  Similarly,
the \emph{palindrome complexity function} $P_{\bf w}(n)$ counts the
number of palindromic factors of ${\bf w}$ of length $n$.  We could
therefore define a \emph{palindrome pair complexity function} $PP_{\bf
  w}(n)$ that counts the number of factors of ${\bf w}$ of length $n$
that are palindrome pairs.  Property {\bf PP} could then be rephrased
as ``$C_{\bf w}(n) = PP_{\bf w}(n)$ for infinitely many $n$.''

\begin{problem}
Investigate the relationships between the functions $C_{\bf w}(n)$,
$P_{\bf w}(n)$, and $PP_{\bf w}(n)$.
\end{problem}

It is known that if ${\bf w}$ has linear factor complexity then its
palindromic complexity is bounded \cite{ABCD03}.  We have already seen
above that this is not true for the palindrome pair complexity.  The
results of Section~\ref{tm} give upper and lower bounds for the
palindrome pair complexity of the Thue--Morse word for certain values
of $n$.

\begin{problem}
Give explicit formulas for the palindrome pair complexity of words
such as the Fibonacci word and the Thue--Morse word.
\end{problem}

Further to the question posed above regarding whether property ${\bf
  PP}$ characterizes some interesting class of words, one might wonder
if, for instance, property {\bf PP} implies $O(n)$ factor complexity.
Unfortunately, this is not the case.  Suppose we perform the
construction of the word ${\bf p}$ in Proposition~\ref{shuffling_pals}
by defining the $w_i$'s as follows:  $w_i = B_iB_i^R$, where $B_i$ is
a \emph{de Bruijn sequence} of order $|p_i|+1$; that is, $B_i$
contains every binary word of length $|p_i|+1$ exactly once.  Then
${\bf p}$ contains at least $2^{|p_i|+1}$ factors of length
$(|p_i|+1)^2$.  If we set $n_i = (|p_i|+1)^2$, then we see that
$C_{\bf p}(n_i) \geq 2^{\sqrt{n_i}}$.  So there are words that are not
so ``nice'' that still have property ${\bf PP}$.

The following would show that property {\bf PP} is a property of
Sturmian words that does not carry over to \emph{episturmian words}.
In principle, it could be resolved using the Walnut Prover, as
described in the proof of Theorem~\ref{tm_pp}.

\begin{conjecture}
The Tribonacci word does not have property {\bf PP}.
\end{conjecture}

In Section~\ref{rich_pp} we already stated the problem of
characterizing the binary words of each length that have the maximum
possible number of distinct palindrome pair factors.
Of course, the most interesting open problem is to resolve the
Frid--Puzynina--Zamboni Conjecture.

\section*{Acknowledgments}
We would like to
thank Luke Schaeffer for his help with using the Walnut Prover.

\end{document}